\documentclass[a4paper,12pt,oneside]{article}
\usepackage{amsmath, amsthm, amssymb, hyperref}

\usepackage{geometry}
\usepackage{enumitem}
\setlist{noitemsep}

\usepackage{listings}
\usepackage{xcolor}

\definecolor{codeblue}{rgb}{0.3,0.3,0.6}
\definecolor{codegreen}{rgb}{0,0.6,0}
\definecolor{codegray}{rgb}{0.5,0.5,0.5}
\definecolor{codepurple}{rgb}{0.58,0,0.82}
\definecolor{backcolour}{rgb}{0.95,0.95,0.92}

\lstdefinestyle{mystyle}{
    backgroundcolor=\color{backcolour},   
    commentstyle=\color{codegreen},
    keywordstyle=\color{codeblue},
    numberstyle=\tiny\color{codegray},
    stringstyle=\color{codepurple},
    basicstyle=\ttfamily\footnotesize,
    breakatwhitespace=false,         
    breaklines=true,                 
    captionpos=b,                    
    keepspaces=true,                 
    numbers=left,                    
    numbersep=5pt,                  
    showspaces=false,                
    showstringspaces=false,
    showtabs=false,                  
    tabsize=2
}

\usepackage{amsthm,amsmath,amsfonts,amssymb}
\usepackage{xcolor}
   \usepackage{geometry}
\usepackage{authblk}

\newgeometry{vmargin={28mm}, hmargin={20mm,17mm}}   

\makeatletter

\hypersetup{
    colorlinks,
    linkcolor={red},
    citecolor={green},
    urlcolor={blue}
}

\newtheorem{theorem}{Theorem}
\newtheorem{lem}[theorem]{Lemma}

\newtheorem{quest}[theorem]{Question}

\newtheorem{defn}[theorem]{Definition}
\newtheorem{maintheorem}{Theorem}

\newtheorem{maincorollary}[maintheorem]{Corollary}

\usepackage{verbatim}
\usepackage{color}
\usepackage{authblk}
\hypersetup{
    colorlinks,
    linkcolor={orange},
    citecolor={magenta},
    urlcolor={blue}
}

\title{Groups with special presentations and star-graph $K_{3,3}$}

\author[1]{Bridgette Amoako\thanks{Email: \texttt{bamoako@uoguelph.ca}}}
\author[2]{Ihechukwu Chinyere\thanks{Corresponding author: \texttt{i.chinyere@up.ac.za; ihechukwu@aims.ac.za}}}
\author[3]{Bernard Bainson\thanks{Email: \texttt{ bboduoku@knust.edu.gh}}}
\affil[1]{Department of Mathematics and Statistics, University of Guelph, Canada}
\affil[2]{Department of Mathematics and Applied Mathematics, University of Pretoria, Hatfield 0028, Pretoria, South Africa}
\affil[3]{Department of Mathematics, Kwame Nkrumah University of Science and
Technology, Ghana}
\begin{document}
\maketitle
\begin{abstract}
We consider a question of Edjvet and Vdovina concerning which groups defined by special presentations are large. For each integer $n \ge 3$, we construct an $n$-generator one-relator presentation whose star graph is the complete bipartite graph $K_{n,n}$; the resulting groups are large and hyperbolic. We also classify concise special presentations with star graph $K_{3,3}$, showing that they are one-relator presentations and that, up to Tietze equivalence, there are exactly twelve that define torsion-free groups. The torsion cases arise precisely as positive powers of the relators in the torsion-free cases, and define pairwise non-isomorphic groups that remain large and hyperbolic.
\end{abstract}

\vspace{1em}

\section{Introduction}

Constrained incidence structures and group presentations serve as a meeting point between combinatorial group theory, incidence geometry, and geometric group theory. In this article, we study group presentations whose \emph{star-graphs} (see Section~\ref{sec:Preliminary}) are the incidence graphs of generalised polygons. This connection was observed by Howie ~\cite{JW89} in his study of the SQ-universality of $T(6)$-groups. In that work, Howie showed that these presentations form the only class of $T(6)$-groups for which SQ-universality could not be settled using his methods. (Recall that a countable group $G$ is \emph{SQ-universal} if every countable group embeds into a quotient of $G$.) Owing to this exceptional status, they became known as \emph{special presentations}, that is, finite presentations with relators of length three whose star-graphs are incidence graphs of finite projective planes. Howie further posed the problem of whether any such group is SQ-universal~\cite[Question~6.11]{JW89}.

\medskip
The theory of special presentations was systematically developed by Edjvet and Howie~\cite{EH88}, who showed that many such groups contain non-abelian free subgroups. Howie’s question was subsequently answered in the negative by Edjvet and Vdovina~\cite{EV10}, who also introduced a broader class of presentations, called \emph{$(m,k)$-special presentations}. These are finite presentations $P=\langle X \mid R\rangle$ satisfying:
\begin{enumerate}
\item the star-graph of $P$ is connected, bipartite, of diameter $m$ and girth $2m$, and each vertex has degree at least $3$;
\item each relator in $R$ has length $k$; and
\item if $m=2$, then $k \ge 4$.
\end{enumerate}
By the classical theorem of Feit and Higman~\cite{FH64}, the diameter parameter of a generalised polygon is restricted to
\(
m\in\{2,3,4,6,8\}.
\)
A group defined by $P$ is said to have a special presentation (or an $(m,k)$-special presentation, to emphasise the combinatorial parameters). Howie’s original special presentations correspond exactly to the $(3,3)$ case. 

\medskip
A further generalisation, the class of $(m,k,\nu)$-special presentations, was introduced in~\cite{CW21, CW23}. In this setting, the star graph is allowed to have $\nu \geq 1$ isomorphic connected components, each of which is the incidence graph of a generalised polygon. In particular, $(m,k)$-special presentations correspond precisely to the case $\nu = 1$, and we retain this notation when $\nu = 1$.

\medskip
In the case of cyclic presentations, the complete classification obtained in~\cite{CW21} shows that only $m=2$ or $m=3$ can occur. Moreover, all $(3,k)$-special presentations in this setting define hyperbolic groups, except for the $(3,3)$-special presentation, which defines a \emph{just-infinite} group~\cite{EV10}, that is, an infinite group all of whose non-trivial quotients are finite. It remains an open problem (see ~\cite[Question 2]{EV10}) whether there exist general construction methods (``example machines'') for producing $(m,k)$-special presentations for $m \in \{4,6,8\}$.

\medskip
The geometric significance of groups defined by special presentations was explored in subsequent work of Vdovina and collaborators. In~\cite{Vdovina99,Vdovina03}, Vdovina constructed finite polygonal complexes whose links are prescribed bipartite graphs, giving hyperbolic complexes whose universal covers are two-dimensional buildings, and the fundamental groups act simply transitively on the vertices of the universal cover. Kangaslampi and Vdovina~\cite{KV} and Carbone, Kangaslampi, and Vdovina~\cite{CKV12} further developed this framework via triangle presentations associated to the smallest generalised quadrangle $GQ(2,2)$, obtaining and classifying groups acting cocompactly and simply transitively on the corresponding triangular hyperbolic buildings. In analogy with this setting, the interest of the present paper lies in the smallest thin case, associated with the incidence graph $K_{3,3}$.

\medskip
Recall that a group is said to be \emph{large} if it has a finite-index subgroup admitting a surjection onto the free group of rank two. From the viewpoint of largeness and SQ-universality, the $(3,3)$-special presentations case stands apart since no group defined by such a presentation is SQ-universal~\cite{EV10}, so largeness does not occur in this setting. For all other $(m,k)$-special presentations with $(m,k)\ne(3,3)$, largeness becomes a natural and largely unexplored property. This motivates the following open problem posed by Edjvet and Vdovina~\cite[Problem~2]{EV10}.

\begin{quest}\label{quest1}
Which groups defined by $(m,k)$-special presentations are large?
\end{quest}

\medskip
For the $23$ torsion-free groups of Kangaslampi and Vdovina~\cite{KV}, SQ-universality follows from~\cite[Theorem~2(v)]{EV10}, but largeness remains unknown. Also, despite the infinite families produced in~\cite{CW21}, no explicit example with $\nu=1$, and in particular with $m=2$, is known to be both large and hyperbolic.  Our first main result provides such a family, thereby providing concrete evidence for to Question~\ref{quest1}.

\begin{maintheorem}\label{main1}
For each integers $\alpha\ge1$ and $n\geq 2$, let
\[
P(n,\alpha)=\langle x_1, x_2, \dots, x_n \mid w_n^\alpha\rangle
\]
be the one-relator presentation with relator defined recursively by
\[
w_1 = x_1, \quad
w_n = w_{n-1}\, x_n \, \prod_{j=1}^{n-1}(x_n x_{n-j}).
\]
Then the star-graph of \(P(n,\alpha)\) is isomorphic to the complete bipartite graph $K_{n,n}$.
\end{maintheorem}
By a standard curvature argument (see~\cite[Proof of Theorem~2]{EV10} and~\cite{GS90}), any group defined by an $(m,k,\nu)$-special presentation satisfying
\begin{equation}\label{ine}
\frac{1}{m}+\frac{2}{k}<1
\end{equation}
is hyperbolic. Moreover, by the theorem of Baumslag--Pride~\cite{BP}, a group admitting a one-relator presentation on $n$ generators has deficiency $n-1$ and is therefore large for $n\ge3$. In view of Theorem~\ref{main1}, the presentations under consideration are $(2,n^2)$-special, and the inequality \eqref{ine} holds for all $n\ge3$. Hence, both hyperbolicity and largeness follow immediately.

\begin{maincorollary}\label{maincor1}
For every integer $n \ge 3$, the presentation $P(n,1)$ is a $(2,n^2)$-special presentation and defines a group that is both large and hyperbolic.
\end{maincorollary}
A group presentation is said to be \emph{redundant} if it contains a \emph{freely redundant relator}, that is, a relator which is freely trivial or freely conjugate to another relator (or to its inverse). A presentation containing no freely redundant relators is called \emph{concise}, and deleting all freely redundant relators gives a \emph{concise refinement}~\cite{CCH81}.

\medskip
Next, we obtain a complete classification of concise $(m,k)$-special presentations (up to Tietze equivalence) whose star-graphs are isomorphic to $K_{3,3}$. These constitute the smallest special presentations in terms of the number of generators.
\newpage
\begin{maintheorem}\label{main2}
A group \(G\) admits a concise \((m,k)\)-special presentation whose star-graph is \(K_{3,3}\) if and only if there exists an integer \(n \ge 1\) such that \(G\) has a presentation of the form
\[
\langle x,y,z \mid r^n \rangle,
\]
where \(r\) is one of the following words:
\begin{align*}
x^2 y^2 z^2 x z y,
&\qquad x^2 y x z^2 y^2 z, \\
x^2 y x^{-1} z y z^{-1} y z, 
&\qquad x^2 y x^{-1} z y z^{-1} y^{-1} z, \\
x^2 y x^{-1} z y^{-1} z y z, 
&\qquad x^2 y x^{-1} z y^{-1} z^{-1} y z, \\
x^2 y z^{-1} x y^{-1} z y z, 
&\qquad x^2 y z^{-1} y z y x^{-1} z, \\
x^2 y z^{-1} y x^{-1} z y z, 
&\qquad x^2 y^{-1} z x y z y z^{-1}, \\
x^2 y^{-1} z y z^{-1} x y z,
&\qquad x^2 y^{-1} x^{-1} z y z^{-1} y z.
\end{align*}
\end{maintheorem}

\medskip
Theorem~\ref{main2} enumerates all concise $(m,k)$-special presentations (up to equivalence) with star-graph $K_{3,3}$, while Corollary~\ref{maincor1} guarantees that the corresponding groups are large and hyperbolic. Among these, the torsion-free cases are characterised as follows.

\begin{maincorollary}\label{maincor2}
Among the group presentations described in Theorem~\ref{main2}, exactly twelve, corresponding to the case $n=1$, define torsion-free groups.
\end{maincorollary}
Corollary \ref{maincor2} follows immediately from \cite[Theorem 3]{MR124384}.

\subsection{Outline}
The paper is organised as follows. Section~\ref{sec:Preliminary} derives structural constraints on $(2,9)$-special relators.  Section~\ref{sec:enumeration} enumerates all admissible relators and reduces them to distinct equivalence classes.  Section~\ref{sec:proofs} contains the proofs of the main Theorem \ref{main1} and Theorem \ref{main2}.  The Appendix \ref{admissible} describes the computational methods, all of which were implemented in \texttt{SageMath}~\cite{sagemath}.

\section{Preliminary results}\label{sec:Preliminary}

In this section, we introduce the basic notions and structural properties that will be used throughout the paper. We begin by recalling the definition of the star-graph associated with a group presentation. We then define $(m,k,\nu)$-special presentations and establish several elementary restrictions on the possible forms of their star-graphs. In particular, we show that any special presentation must involve at least three generators, and we determine the parameters of those whose star-graph is isomorphic to $K_{3,3}$.

\begin{defn}[star-graph]\label{def:stargraph}
Let $P = \langle X \mid R \rangle$ be a group presentation. The \emph{star-graph} of $P$ is the undirected labelled graph $\Gamma$ with vertex set
\(
V(\Gamma) = X^{\pm 1}.
\)
An edge joins two vertices $x_1, x_2 \in V(\Gamma)$ whenever there exists a relator of the form
\(x_1 x_2^{-1} u\) or \(x_2 x_1^{-1} u^{-1}\) 
in  $\widetilde{R}$, where $\widetilde{R}$ (also known as the \emph{symmetrised closure} of $R$)  denotes the set consisting of all elements of $R \cup R^{-1}$ together with all of their cyclic permutations.
\end{defn}

\begin{defn}[{\cite[Definition~2.1]{CW21}}]\label{def:mknspecial}
Let $m \ge 2$, $k \ge 3$, and $\nu \ge 1$ be integers. A finite group presentation $P = \langle X \mid R \rangle$ is said to be \emph{$(m,k,\nu)$–special} if the following conditions hold:
\begin{enumerate}
    \item The star-graph $\Gamma$ of $P$ has $\nu$ isomorphic connected components, each of which is bipartite of diameter $m$ and girth $2m$, and in which every vertex has degree at least $3$.
    \item Each relator $r \in R$ has length $k$.
    \item If $m = 2$, then $k \ge 4$.
\end{enumerate}
\end{defn}
We will refer to an $(m,k,\nu)$--special presentation simply as a special presentation whenever the parameters need not be specified.
\begin{lem}\label{lem:threegens}
Suppose $P = \langle X \mid R \rangle$ is an $(m,k,\nu)$-special presentation. Then $|X| \ge 3$.
\end{lem}

\begin{proof}
Let $\Gamma$ denote the star-graph of $P$, and let $V(\Gamma)$ be its vertex set. Then
\[
V(\Gamma)=X \cup X^{-1},
\quad \text{so } |V(\Gamma)|=2|X|.
\]
If $|X|\le 2$, then $|V(\Gamma)|\le 4$. Since $\Gamma$ is bipartite, its vertex set admits a bipartition $(A, B)$ with $|B|\le 2$. Every vertex of $\Gamma$ is adjacent only to vertices in the opposite part of the bipartition, so any vertex in $A$ has degree at most $|B|\le 2$. Hence $\Gamma$ contains a vertex of degree at most $2$. This contradicts the assumption that every vertex has degree at least $3$. Therefore $|X|\ge 3$.
\end{proof}

\medskip
If a one–relator presentation $P=\langle X \mid r\rangle$ has star-graph isomorphic to $K_{3,3}$, then for every $n\in\mathbb{N}$ the presentation $\langle X \mid r^{\,n}\rangle$ also has star-graph $K_{3,3}$. Indeed, replacing $r$ by a positive power does not change the set of length--two subwords appearing in the symmetrised closure, and hence does not alter the edge set of the star-graph.  Also, freely redundant relators contribute no new edges to the star-graph, so replacing a presentation by a concise refinement leaves the star-graph unchanged. Consequently, in classifying presentations whose star-graph is $K_{3,3}$, we may assume without loss of generality that the relator $r$ is not a proper power and that the presentation is concise.

\begin{lem}\label{lem:k33}
Suppose $P = \langle X \mid R \rangle$ is a concise $(m,k,\nu)$-special presentation whose star-graph is $K_{3,3}$. Then $R=\{r\}$, $\ell(r)=k=9$, and $\nu=1$.
\end{lem}

\begin{proof}
Let $P=\langle X \mid R\rangle$ be a concise $(m,k,\nu)$–special presentation whose star-graph is $K_{3,3}$. Then, $m=2$ and so every relator has length $k\ge4$ by Definition~\ref{def:mknspecial}.  The graph $K_{3,3}$ has  exactly $9$ edges. In the construction of the star-graph, each relator of length $k$ contributes precisely $k$ edges. Therefore
\[
k \cdot |R| \;=\; 9.
\]
Since $k\ge4$, the only possible factorisation is $k=9$ and $|R|=1$. Hence $R=\{r\}$ with $\ell(r)=k=9$.  Finally, $K_{3,3}$ is connected, so the star-graph has a single component and hence $\nu=1$.
\end{proof}
A group $G$ is called \emph{Hopfian} if every surjective endomorphism of $G$ is an automorphism. Before ending this section, we give a result about one-relator groups which will be useful later.
\begin{lem}\label{hopf}
Let
\(
H_1=\langle X \mid r\rangle\) and \(H_2=\langle Y \mid s\rangle
\) be
Hopfian groups with $r,s$ not proper powers. For $n\ge2$, set
\(
G_1=\langle X \mid r^n\rangle\) and \(G_2=\langle Y \mid s^n\rangle.
\)
If $H_1 \not\cong H_2$, then
\(
G_1 \not\cong G_2.
\)
\end{lem}
\begin{proof}
Assume for contradiction that $H_1 \not\cong H_2$ but $G_1 \cong G_2$. Each $G_i$ admits a natural epimorphism onto $H_i$ by adding the relation $r=1$ or $s=1$, respectively. Transporting the map $G_1 \twoheadrightarrow H_1$ across an isomorphism $G_1 \cong G_2$ gives an epimorphism
\(
\psi: G_2 \twoheadrightarrow H_1.
\)
In $G_2$, the element $s$ has finite order $n$, and every torsion element is conjugate to a power of $s$. Since $H_1$ is torsion-free, $\psi(s)=1$. Hence $\ker(\phi)$ contains the normal closure of $s$, and therefore $\psi$ factors through the quotient
\(
H_2 = G_2/\!\langle\!\langle s\rangle\!\rangle,
\)
giving an epimorphism from $H_2$ to  $H_1$. By symmetry, we also obtain an epimorphism from $H_1$ to  $H_2$. Hence, $H_1$ and $H_2$ admit mutual epimorphisms. Since $H_1$ and $H_2$ are Hopfian, the compositions
\[
H_1 \twoheadrightarrow H_1,\quad H_2 \twoheadrightarrow H_2
\]
are automorphisms. Hence $H_1 \cong H_2$, contradicting the hypothesis. Therefore $G_1 \not\cong G_2$.
\end{proof}

\section{Relators for \texorpdfstring{$(2,9)$}{(2,9)}-special presentations}\label{sec:enumeration}

Let $F(x,y,z)$ denote the free group of rank $3$ on the generators $x,y,z$. In this section, we classify all cyclically reduced words of length $9$ in $F(x,y,z)$ that define $(2,9)$-special presentations whose star-graphs are isomorphic to $K_{3,3}$. By Lemma~\ref{lem:k33}, each such presentation is determined by a cyclically-reduced relator of length $9$. Our objective is to enumerate all admissible relators and then refine this enumeration to obtain the isomorphism classes of the corresponding one-relator groups.

\medskip
The classification proceeds in two steps. First, we impose the combinatorial constraints coming from the $K_{3,3}$ star-graph condition and use a computer search (Appendix~\ref{admissible}) to generate all admissible relators. This gives a finite list of words, which is then reduced modulo cyclic permutation, re-labelling and inversion.

\begin{lem}\label{lem:enumeration}
The set of admissible relators $R$ decomposes as
\(
R = \bigcup_{i=1}^{12} R_i,
\)
where each $R_i$ is given explicitly below.
\begin{align*}
R_1 &= \{x^2 y x z y^2 z^2, \; x^2 y z x z^2 y^2, \; x^2 y^2 x z y z^2, \; x^2 y^2 z y x z^2, \; x^2 y^2 z^2 y x z, \; x^2 y^2 z^2 x z y\}, \\
R_2 &= \{x^2 y z y^2 x z^2, \; x^2 y z^2 x z y^2, \; x^2 y z^2 y^2 x z, \; x^2 y^2 x z^2 y z, \; x^2 y^2 z x z^2 y, \; x^2 y x z^2 y^2 z\}, \\
R_3 &= \{x^2 y x^{-1} z y z^{-1} y z, \; x^2 y z y^{-1} z y x^{-1} z\}, 
R_4 = \{x^2 y x^{-1} z y z^{-1} y^{-1} z, \; x^2 y z^{-1} y^{-1} z y x^{-1} z\}, \\
R_5 &= \{x^2 y x^{-1} z y^{-1} z y z, \; x^2 y z y z^{-1} y x^{-1} z\}, 
R_6 = \{x^2 y x^{-1} z y^{-1} z^{-1} y z, \; x^2 y z y^{-1} z^{-1} y x^{-1} z\}, \\
R_7 &= \{x^2 y z^{-1} x y^{-1} z y z, \; x^2 y z y z^{-1} x y^{-1} z\}, 
R_8 = \{x^2 y z^{-1} y z y x^{-1} z, \; x^2 y x^{-1} z y z y^{-1} z\}, \\
R_9 &= \{x^2 y z^{-1} y x^{-1} z y z, \; x^2 y z y x^{-1} z y^{-1} z\}, 
R_{10} = \{x^2 y^{-1} z x y z y z^{-1}, \; x^2 y^{-1} z y z x y z^{-1}\}, \\
R_{11} &= \{x^2 y^{-1} z y z^{-1} x y z, \; x^2 y z x y^{-1} z y z^{-1}\}, 
R_{12} = \{x^2 y^{-1} x^{-1} z y z^{-1} y z, \; x^2 y z y^{-1} z y x^{-1} z^{-1}\}.
\end{align*}
\end{lem}

The proof is computational and uses structural constraints imposed by the $K_{3,3}$ star-graph condition. Each admissible relator must contain each of $x^{\pm1}, y^{\pm1}, z^{\pm1}$ exactly three times. By replacing a generator $t \in \{x,y,z\}$ with $t^{-1}$ if necessary, we may further assume that the number of occurrences of $t$ is at least the number of occurrences of $t^{-1}$. Note that the $K_{3,3}$ condition forces the occurrence of $t^2$, and not $t^3$, as a subword for some $t \in \{x,y,z\}$. We normalise by cyclic permutation, re-labelling, and inversion so that each word begins with $x^2 y^{\pm 1}$. Indeed, after fixing an initial segment $x^2$, the next letter cannot be $x^{\pm1}$ without violating cyclic reduction of $K_{3,3}$ condition, and hence must be either $y^{\pm1}$ or $z^{\pm1}$. In the latter case, we apply the automorphism of the free group that swaps $y$ and $z$, reducing to the former situation. The Sage code in Appendix~\ref{admissible} enumerates all cyclically reduced words of length $9$ that satisfy these constraints, and then selects those whose star-graphs have girth $4$ and diameter $2$. The output consists of $32$ words, which form the $12$ equivalence classes of Lemma~\ref{lem:enumeration}.

\medskip
To obtain a complete classification, we next introduce additional automorphisms of $F(x,y,z)$. For each $t \in \{x,y,z\}$, let $\phi_t$ be the automorphism fixing $t$ and swapping the other two generators, and let $\rho_t$ be the automorphism inverting $t$ while fixing the remaining generators. These automorphisms, together with cyclic permutation and inversion, generate all identifications needed to relate relators within each class $R_i$.

\begin{lem}\label{lem:ri-isomorphism}
Let $R = \bigcup_{i=1}^{12} R_i$ be the decomposition given in Lemma~\ref{lem:enumeration}. 
For each $i = 1, \dots, 12$ and any $w, w' \in R_i$, the one-relator groups defined by
\(
\langle x, y, z \mid w \rangle\) and  
\(\langle x, y, z \mid w' \rangle
\)
are isomorphic via compositions of the automorphisms $\phi_t$, $\rho_t$, together with cyclic permutation and inversion.
\end{lem}

\begin{proof}
For each class $R_i$, Table~\ref{tab:appendix-maps} in the Appendix gives an explicit composition of automorphisms, inversion, and cyclic permutation sending any relator in $R_i$ to a chosen representative $w_0$. This shows that all presentations within a given class define isomorphic one-relator groups.
\end{proof}

\section{Main}\label{sec:proofs}

In this section, we prove Theorem~\ref{main1} and Theorem~\ref{main2}. 

\medskip
We begin with the proof of Theorem~\ref{main1}, which constructs an explicit family of one-relator presentations and determines their star-graphs. 
\begin{proof}[Proof of Theorem~\ref{main1}]
For integers $\alpha\ge 1$ and $n\geq 2$ consider the one--relator presentation
\[
P(n,\alpha)=\langle x_1,\dots,x_n\mid w_n^\alpha\rangle,
\quad 
 w_1=x_1,\quad 
 w_n=w_{n-1}\,x_n\prod_{j=1}^{n-1}(x_nx_{n-j})\ (n\ge 2).
\]
We show that the star-graph of $P_n$ is the complete bipartite graph $K_{n,n}$. Recall that for a cyclically reduced relator $w$, each cyclic subword $ab$ contributes an oriented edge recorded as $(a,b^{-1})$.  Since $w_n$ is a positive word, all such pairs are of the form $(x_i,x_j^{-1})$.  Hence, it suffices to prove that the cyclic adjacencies in $w_n$ realise every ordered pair $(x_i,x_j^{-1})$ with $1\le i,j\le n$ exactly once. The argument proceeds by induction. 

\medskip
For $n=2$ we have $w_2=x_1x_2x_2x_1$, whose cyclic adjacencies give precisely the four pairs $(x_i,x_j^{-1})$ with $i,j\in\{1,2\}$.  Assume the statement holds for $n-1\ge 2$.  Write
\[
w_n = w_{n-1}\,T,
\quad 
T = x_n\,(x_nx_{n-1})\,(x_nx_{n-2})\cdots(x_nx_1).
\]
The cyclic adjacencies contributed by $w_{n-1}$ realise all pairs $(x_i,x_j^{-1})$ with $1\le i,j\le n-1$.  The new adjacencies arising from $T$ and from the boundary between $w_{n-1}$ and $T$ contributes exactly the pairs
\[
(x_n,x_n^{-1}),\quad (x_n,x_j^{-1}),\quad (x_j,x_n^{-1}) \quad (1\le j\le n-1).
\]
These are disjoint from the previous $(n-1)^2$ pairs and account for $2n-1$ new pairs.  Hence, the total number of distinct pairs realised by $w_n$ is
\(
(n-1)^2+(2n-1)=n^2,
\)
which equals the total number of ordered pairs $(x_i,x_j^{-1})$ with $1\le i,j\le n$.  Therefore the star-graph of $P(n, \alpha)$ is $K_{n,n}$.
\end{proof}

Finally, we turn to Theorem~\ref{main2}, where the goal is to distinguish the resulting groups up to isomorphism. This requires a finer invariant: for a finitely presented group, the collection of abelianisations of its finite-index subgroups is an isomorphism invariant. We use this computationally, via \texttt{SageMath}, to separate the groups arising in our classification.

\begin{proof}[Proof of Theorem~\ref{main2}]
Let $G_i=\langle x,y,z\mid r_i\rangle$ for $1\le i\le12$, where $r_i\in R_i$ as given in Lemma~\ref{lem:enumeration}. By Lemma~\ref{lem:ri-isomorphism}, the isomorphism type of $G_i$ depends only on the equivalence class $R_i$ and not on the particular choice of relator $r_i\in R_i$. To prove that these groups are pairwise non--isomorphic, we compare the abelianisations $\mathrm{Ab}(X)=X/[X,X]$ of low--index subgroups. If $G\cong H$, then for each $k$ the multisets
\[
\bigl\{\,\mathrm{Ab}(U)\mid U\le G,\ [G:U]=k\,\bigr\}
\quad\text{and}\quad
\bigl\{\,\mathrm{Ab}(V)\mid V\le H,\ [H:V]=k\,\bigr\}
\]
coincide.  Hence, to show $G_i\not\cong G_j$, it suffices to find an index $k$ such that one of these multisets differs. All computations were carried out in \texttt{SageMath} using low--index subgroup enumeration.  We record the separating invariants used in each case.

\medskip
\noindent\emph{The group $G_1$.}  
Among the index--$3$ subgroups of $G_1$, there is one with abelianisation $\mathbb Z^4\oplus\mathbb Z_9$, which does not occur for any index--$3$ subgroup of $G_j$ with $j>1$.  Hence $G_1\not\cong G_j$ for all $j>1$.

\medskip
\noindent\emph{The group $G_2$.}  
An index--$3$ subgroup of $G_2$ has abelianisation $\mathbb Z^4\oplus\mathbb Z_3^2$, which does not occur for any index--$3$ subgroup of $G_j$ with $j>2$.  Hence $G_2\not\cong G_j$ for all $j>2$.

\medskip
\noindent\emph{The group $G_3$.}  
Among index--$3$ subgroups, $G_3$ has one with abelianisation $\mathbb Z^4\oplus\mathbb Z_2^2$, which does not occur for $G_j$ with $j\in\{4,5,6,7,8,10,11\}$.  
At index $5$, the group $G_9$ has a subgroup with abelianisation $\mathbb Z^6\oplus\mathbb Z_{11}$, which does not occur for $G_3$, so $G_3\not\cong G_9$.  
Finally, $G_{12}$ has two index--$5$ subgroups with abelianisation $\mathbb Z^6\oplus\mathbb Z_2$, whereas $G_3$ has only one.  Hence $G_3$ is distinct from all $G_j$ with $j>3$.

\medskip
\noindent\emph{The group $G_4$.}  
Among index--$3$ subgroups, $G_4$ has one with abelianisation $\mathbb Z^4\oplus\mathbb Z_7$, which does not occur for $G_j$ with $j\in\{5,6,7,8,9,10,12\}$.  
Another index--$3$ subgroup has abelianisation $\mathbb Z^4\oplus\mathbb Z_2$, which does not occur for $G_{11}$.  Hence $G_4\not\cong G_j$ for all $j>4$.

\medskip
\noindent\emph{The group $G_5$.}  
An index--$3$ subgroup has abelianisation $\mathbb Z^5$, which does not occur for $G_j$ with $j\in\{6,7,9,10,12\}$.  
At index $4$, $G_5$ has a subgroup with abelianisation $\mathbb Z^5\oplus\mathbb Z_5$, which does not occur for $G_8$.  
Moreover, $G_{11}$ has an index--$3$ subgroup with abelianisation $\mathbb Z^4\oplus\mathbb Z_7$, which does not occur for $G_5$.  
Hence $G_5\not\cong G_j$ for all $j>5$.

\medskip
\noindent\emph{The group $G_6$.}  
An index--$3$ subgroup has abelianisation $\mathbb Z^4\oplus\mathbb Z_2$, which does not occur for $G_j$ with $j\in\{8,9,11,12\}$.  
An index--$4$ subgroup has abelianisation $\mathbb Z^5\oplus\mathbb Z_5$, which distinguishes $G_6$ from $G_7$ and $G_{10}$.  
Hence $G_6\not\cong G_j$ for all $j>6$.

\medskip
\noindent\emph{The group $G_7$.}  
An index--$3$ subgroup has abelianisation $\mathbb Z^4\oplus\mathbb Z_2$, which does not occur for $G_j$ with $j\in\{8,9,11,12\}$.  
At index $5$, $G_7$ has three subgroups with abelianisation $\mathbb Z^6\oplus\mathbb Z_2$, whereas $G_{10}$ has only two.  
Hence $G_7\not\cong G_j$ for all $j>7$.

\medskip
\noindent\emph{The group $G_8$.}  
An index--$3$ subgroup has abelianisation $\mathbb Z^5$, which does not occur for $G_j$ with $j\in\{9,10,12\}$.  
Moreover, $G_{11}$ has an index--$3$ subgroup with abelianisation $\mathbb Z^4\oplus\mathbb Z_7$, which does not occur for $G_8$.  
Hence $G_8\not\cong G_j$ for all $j>8$.

\medskip
\noindent\emph{The group $G_9$.}  
An index--$3$ subgroup has abelianisation $\mathbb Z^4\oplus\mathbb Z_2^2$, which distinguishes $G_9$ from $G_{10}$ and $G_{11}$.  
At index $5$, $G_9$ has a subgroup with abelianisation $\mathbb Z^6\oplus\mathbb Z_{11}$, which does not occur for $G_{12}$.  
Hence $G_9\not\cong G_j$ for all $j>9$.

\medskip
\noindent\emph{The group $G_{10}$.}  
An index--$3$ subgroup has abelianisation $\mathbb Z^4\oplus\mathbb Z_2$, which distinguishes $G_{10}$ from $G_{11}$ and $G_{12}$.  
Hence $G_{10}\not\cong G_j$ for all $j>10$.

\medskip
\noindent\emph{The group $G_{11}$.}  
An index--$3$ subgroup has abelianisation $\mathbb Z^4\oplus\mathbb Z_7$, which does not occur for $G_{12}$.  
Hence $G_{11}\not\cong G_{12}$.

\medskip
Since every pair $G_i,G_j$ with $i\neq j$ is separated by at least one low--index abelianisation invariant, the twelve groups are pairwise non--isomorphic.  The rest of the proof follows from Lemma \ref{hopf} since the groups $G_i$ are hyperbolic, hence Hopfian by \cite[Corollary 2.9]{MR4623546} (see also \cite{MR4140867,MR1660337}).
\end{proof}

\section*{Acknowledgments}
Much of the content of this article is based on the first author's thesis completed at the African Institute for Mathematical Sciences (AIMS) in Ghana (see~\cite{thesis}). The second author gratefully acknowledges the support of the London Mathematical Society through a Scheme~5 ``Collaborations with Developing Countries'' grant (Reference: 52107), which enabled travel to AIMS in 2023 for further discussions and collaboration on this work.

\bibliographystyle{alpha}
\bibliography{Refs}

\newpage
\appendix
\section{Automorphisms identifying relators}\label{app:automorphisms}
\begin{table}[h!]
\centering
\begin{tabular}{|c|c|c|}
\hline
$R_i$ & Word $w$ & Composition to $w_0$ \\
\hline
$R_1$ & $x^2 y^2 z^2 x z y$ & $w_0$  \\
      & $x^2 y x z y^2 z^2$ & $\phi_x \circ \phi_y$, cyclic permutation \\
      & $x^2 y^2 x z y z^2$ & $\rho_x \circ \rho_y \circ \rho_z$, invert, $\phi_z$, cyclic permutation \\
      & $x^2 y^2 z y x z^2$ & $\phi_x \circ \phi_z$, cyclic permutation \\
      & $x^2 y^2 z^2 y x z$ & $\phi_y \circ \rho_x \circ \rho_y \circ \rho_z$, cyclic permutation \\
\hline
$R_2$ & $x^2 y x z^2 y^2 z$ & $w_0$  \\
      & $x^2 y z y^2 x z^2$ & $\phi_x \circ \phi_y \circ \rho_x \circ \rho_y \circ \rho_z$, cyclic permutation \\
      & $x^2 y z^2 x z y^2$ & $\phi_x \circ \phi_z$, cyclic permutation \\
      & $x^2 y z^2 y^2 x z$ & $\rho_x \circ \rho_y \circ \rho_z$, invert, $\phi_x$, cyclic permutation \\
      & $x^2 y^2 x z^2 y z$ & $\phi_y$, cyclic permutation \\
      & $x^2 y^2 z x z^2 y$ & $\phi_x \circ \phi_z \circ \rho_x \circ \rho_y \circ \rho_z$, cyclic permutation \\
\hline
$R_3$ & $x^2 y x^{-1} z y z^{-1} y z$ & $w_0$  \\
      & $x^2 y z y^{-1} z y x^{-1} z$ & $\phi_x\circ\rho_z \circ \rho_y \circ \rho_x$, invert, cyclic permutation \\
\hline
$R_4$ & $x^2 y x^{-1} z y z^{-1} y^{-1} z$ & $w_0$  \\
      & $x^2 y z^{-1} y^{-1} z y x^{-1} z$ & $\phi_x\circ\rho_z \circ \rho_y \circ \rho_x$, invert, cyclic permutation \\
\hline
$R_5$ & $x^2 y x^{-1} z y^{-1} z y z$ & $w_0$  \\
      & $x^2 y z y z^{-1} y x^{-1} z$ & $\phi_x\circ\rho_z \circ \rho_y \circ \rho_x$, invert, cyclic permutation \\
\hline
$R_6$ & $x^2 y x^{-1} z y^{-1} z^{-1} y z$ & $w_0$  \\
      & $x^2 y z y^{-1} z^{-1} y x^{-1} z$ & $\phi_x\circ\rho_z \circ \rho_y \circ \rho_x$, invert, cyclic permutation \\
\hline
$R_7$ & $x^2 y z^{-1} x y^{-1} z y z$ & $w_0$  \\
      & $x^2 y z y z^{-1} x y^{-1} z$ & $\phi_x\circ\rho_z \circ \rho_y \circ \rho_x$, invert, cyclic permutation \\
\hline
$R_8$ & $x^2 y z^{-1} y z y x^{-1} z$ & $w_0$  \\
      & $x^2 y x^{-1} z y z y^{-1} z$ & $\phi_x\circ\rho_z \circ \rho_y \circ \rho_x$, invert, cyclic permutation \\
\hline
$R_9$ & $x^2 y z^{-1} y x^{-1} z y z$ & $w_0$  \\
      & $x^2 y z y x^{-1} z y^{-1} z$ & $\phi_x\circ\rho_z \circ \rho_y \circ \rho_x$, invert, cyclic permutation \\
\hline
$R_{10}$ & $x^2 y^{-1} z x y z y z^{-1}$ & $w_0$  \\
         & $x^2 y^{-1} z y z x y z^{-1}$ & $\phi_x\circ\rho_z \circ \rho_y \circ \rho_x$, invert, cyclic permutation \\
\hline
$R_{11}$ & $x^2 y^{-1} z y z^{-1} x y z$ & $w_0$  \\
         & $x^2 y z x y^{-1} z y z^{-1}$ & $\phi_x\circ\rho_z \circ \rho_y \circ \rho_x$, invert, cyclic permutation \\
\hline
$R_{12}$ & $x^2 y^{-1} x^{-1} z y z^{-1} y z$ & $w_0$  \\
         & $x^2 y z y^{-1} z y x^{-1} z^{-1}$ & $\phi_x\circ\rho_z \circ \rho_y \circ \rho_x$, invert, cyclic permutation \\
\hline
\end{tabular}
\caption{Automorphisms identifying the relators in each class $R_i$ with the representative $w_0$. We use the convention that $(f\circ g)(x)=g(f(x))$.}
\label{tab:appendix-maps}
\end{table}
\newpage
\section{Computational enumeration of $(2,9)$-special relators}\label{admissible}
\begin{lstlisting}[language=Python]
from sage.graphs.graph import Graph
#Generates all words of length `l` from a given alphabet.
def words(alphabet, l):
    if l == 0:
        give []
    else:
        for word in words(alphabet, l - 1):
            for letter in alphabet:
                give word + [letter]
                
#Constructs a star-graph for a given word `g`.
def stargraph(g):
    G = Graph(loops=True)
    G.add_vertices(['x','y','z','X','Y','Z'])    
    for i in range(len(g)):
        prev = g[i-1]       
        cur = g[i]          
        if cur.isupper():
            G.add_edge((prev, cur.lower()))
        else:
            G.add_edge((prev, cur.upper()))
    return G
    
#Generates words of length 9 from the alphabet {x,y,z,X,Y,Z} & filters them 
alphabet = ['x','y','z','X','Y','Z']
word_length = 9
for word in words(alphabet, word_length):
    if not (word[0:3] == ['x','x','y'] or word[0:3] == ['x','x','Y']):
        continue
    if not (word.count('x') + word.count('X') == 3 and
            word.count('y') + word.count('Y') == 3 and
            word.count('z') + word.count('Z') == 3 and
            word.count('y') > word.count('Y') and
            word.count('z') > word.count('Z')):
        continue
    SG = stargraph(word)
    if SG.girth() == 4 and SG.diameter() == 2:
        print(''.join(word))
        
#Running the code above gives this list of 32 words
words_list = [
    "xxyxzyyzz","xxyxzzyyz","xxyyxzyzz","xxyyxzzyz","xxyyzxzzy",
    "xxyyzyxzz","xxyyzzxzy","xxyyzzyxz","xxyzxzzyy","xxyzxYzyZ",
    "xxyzyyxzz","xxyzyXzYz","xxyzyZxYz","xxyzyZyXz","xxyzzxzyy",
    "xxyzzyyxz","xxyzYzyXz","xxyzYzyXZ","xxyzYZyXz","xxyXzyzYz",
    "xxyXzyZyz","xxyXzyZYz","xxyXzYzyz","xxyXzYZyz","xxyZxYzyz",
    "xxyZyzyXz","xxyZyXzyz","xxyZYzyXz","xxYzxyzyZ","xxYzyzxyZ",
    "xxYzyZxyz","xxYXzyZyz"
]
\end{lstlisting}
\end{document}